\documentclass[12pt, a4paper]{amsart}
\usepackage{amssymb}
\usepackage{amsmath}
\usepackage{bbm}
\usepackage{color}
\usepackage{verbatim}
\usepackage{pictex}

\definecolor{darkred}{rgb}{0.75,0,0.3}

\newcommand\al{\alpha}
\newcommand\Aff{\operatorname{\sf Aff}}
\newcommand\AND{\quad\text{and}\quad}
\newcommand\Aut{\operatorname{\sf Aut}}
\newcommand\C{\mathbb C}
\newcommand\de{\delta}
\newcommand\dg{\mathsf{deg}}
\newcommand\hor{\mathfrak{h}}
\newcommand\Hor{\mathfrak{H}}
\newcommand\la{\lambda}
\newcommand\mm{\mathsf m}

\newcommand\ol{\overline}
\newcommand\R{\mathbb R}
\newcommand\res{\mathsf{res}}
\newcommand\spec{\mathsf{spec}}
\newcommand\T{\mathbb T}
\newcommand\uno{\mathbf{1}}
\newcommand\wh{\widehat}
\newcommand\wt{\widetilde}
\newcommand\Z{\mathbb Z}

\renewcommand{\le}{\leqslant}

\renewcommand{\ge}{\geqslant}

\numberwithin{equation}{section}

\newtheoremstyle{mythm}
  {9pt}
  {9pt}
  {\itshape}
  {0pt}
  {\bfseries}
  {}
  { }
  {\thmnumber{(#2)}\thmname{ #1}\thmnote{ #3}}

\newtheoremstyle{mydef}
  {9pt}
  {9pt}
  {\normalfont}
  {0pt}
  {\bfseries}
  {}
  { }
  {\thmnumber{(#2)}\thmname{ #1}\thmnote{ #3}}

\theoremstyle{mythm}
\newtheorem{thm}[equation]{Theorem.}
\newtheorem{pro}[equation]{Proposition.}
\newtheorem{lem}[equation]{Lemma.}
\newtheorem{cor}[equation]{Corollary.}

\theoremstyle{mydef}
\newtheorem{con}[equation]{Convention.}

\newtheorem{rmk}[equation]{Remark.}
\newtheorem{rmks}[equation]{Remarks.}

\begin{document}$\,$ \vspace{-1truecm}
\title{Multiple boundary representations of $\la$-harmonic functions on trees}
\author{\bf Massimo A. Picardello, Wolfgang Woess}
\address{\parbox{.8\linewidth}{Dipartimento di Matematica\\ 
Universit\`a di Roma ``Tor Vergata''\\
I-00133 Rome, Italy}}
\email{picard@axp.mat.uniroma2.it}

\address{\parbox{.8\linewidth}{Institut f\"ur Diskrete Mathematik,\\ 
Technische Universit\"at Graz,\\
Steyrergasse 30, A-8010 Graz, Austria}}
\email{woess@tugraz.at}

\begin{abstract}
We consider a countable tree $T$, possibly having vertices with infinite degree,
and an arbitrary stochastic nearest neighbour transition
operator $P$. We provide a boundary integral representation for
general eigenfunctions of $P$ with eigenvalue $\lambda \in C$, under the condition
that the oriented edges can be equipped with complex-valued weights satisfying
three natural axioms. These axioms guarantee that one can construct a $\lambda$-Poisson kernel.
The boundary integral is with respect to distributions, that is, elements in the
dual of the space of locally constant functions. Distributions are interpreted as finitely
additive complex measures. In general, they do not extend to $\sigma$-additive 
measures: for this extension, 
a summability condition over disjoint boundary arcs is required. Whenever $\lambda$ is in the resolvent of $P$
as a self-adjoint operator on a naturally associated $\ell^2$-space and the 
diagonal elements of the resolvent (``Green function'') do not vanish
at $\lambda$, one can use the ordinary edge weights corresponding to the Green function
and obtain the ordinary $\la$-Martin kernel.
 
We then consider the case when $P$ is invariant under a transitive
group action. In this situation, we study the phenomenon that in addition to the 
$\la$-Martin kernel, there may be further choices for the edge weights which give rise
to another $\la$-Poisson kernel with associated integral representations.
In particular, we compare the resulting distributions on the boundary.

The material presented here is closely related to the contents of our ``companion'' paper
\cite{PiWo}.
 \end{abstract}

\maketitle

\markboth{{\sf M. A. Picardello and W. Woess}}
{{\sf Multiple boundary representations on trees}}
\baselineskip 15pt

\section{Introduction}\label{sec:intro}

Let $T$ be a countable tree, i.e., a connected graph without cycles. We allow vertices
with infinite degree, but for simplicity, we exclude leaves (vertices with degree $1$).
Here, the degree $\dg(x)$ of a vertex $x$ is the number of its neighbours.
We tacitly identify $T$ with its vertex set.

On $T$, we consider the stochastic transition matrix 
$P = \bigl(p(x,y)\bigr)_{x,y \in T}$ of a nearest neighbour random walk.
This means that $p(x,y) > 0$ if and only if $x \sim y$ (i.e., $x$ and $y$ 
are neighbours).  $P$ acts on functions $f: T \to \C$ by
\begin{equation}\label{eq:Pf}
Pf(x) =  \sum_y p(x,y)f(y)\,,
\end{equation}
where in case when $\dg(x) = \infty$ we postulate that the sum converges absolutely.
For $\la \in \C$, a \emph{$\la$-harmonic function} is a function $h: T \to \C$
which satisfies $Ph = \la \cdot h\,$.

For ``good'' values of $\la$, every $\la$-harmonic function has a boundary integral
representation over the geometric boundary at infinity of the tree. This is
analogous to the Poisson integral formula for classical harmonic functions on the open 
unit disk, where the boundary is the unit circle. The Poisson kernel of the disk has
to be replaced by the \emph{$\la$-Martin kernel,} and the integral is with respect to a
\emph{distribution} on the boundary.  The good values include in particular
$\la=1$, when the random walk is transient.
More generally, they comprise at least all $\la \in \C$ where $|\la| > \rho$  with 
$\rho=\rho(P)$, the \emph{spectral radius} of the random walk (the definitions will be given in more detail further on).
For \emph{positive} $\la$-harmonic functions -- whose existence necessitates that
$\la \ge \rho$ is real --  the representing distribution on the
boundary is a finite ($\sigma$-additive) Borel measure.  

The results that we have mentioned in this last paragraph are due to
{\sc Cartier}~\cite{Ca} for the case when $\la \ge \rho$ and the tree is
locally finite, and the extension to the non-locally finite case can be found
in the book of {\sc Woess}~\cite[Ch. 9]{W-Markov}. For general complex 
$\la$, these results are proved in our recent paper \cite{PiWo}, when $\la$ is in
the resolvent set of $P$ and the diagonal elements of the \emph{Green kernel} (Green function)
do not vanish at $\la$. This was preceded by a result of 
{\sc Fig\`a-Talamanca and Steger}~\cite{FiSt}
for the locally fininte case, when $P$ is the transition matrix of a group invariant random walk 
on a free group, or a close relative of that group.

All this comprises the long known example of the \emph{simple random walk} on $T = \T_q\,$,
the regular tree with degree $q+1 \ge 3$, where $p(x,y) = 1/(q+1)$ when $x \sim y$.
In this case, it follows from the results of 
{\sc Mantero and Zappa}~\cite{MaZa} that,
besides the ordinary $\la$-Martin kernel, there is a second kernel which gives rise to 
a boundary integral representation of $\la$-harmonic functions. Indeed, this plays
an important role  in the context of the representation theory of free groups. 
Since then, this phenomenon has remained the object of repeated discussions,
in particular between the first author and David Singman (George Mason University, Fairfax).

The purpose of the present note is to shed more light on these
multiple boundary integral representation by approaching them from a wider viewpoint. Thereby, 
part of our presentation lays out in detail several proofs which 
take up and generalize previous work. 

We first (\S \ref{sec:bdry}) recall the construction of the boundary at infinity $\partial T$ of $T$ and the 
corresponding compactification. We introduce distributions on $\partial T$ and explain
how locally constant functions on $\partial T$ are integrated against a distribution.

Then (\S \ref{sec:intrep}) we start with an arbitrary $\la \in \C$ and put weights on the 
oriented edges
of $T$. They are required to satisfy certain axioms (this might not be possible for all $\la$)
and then they can be used to define a general $\la$-potential
kernel and subsequently a $\la$-Poisson kernel $k(x,y)$, $x,y \in T$. This kernel extends 
in the second variable to a locally
constant function on $\partial T$, and we use it to prove a general Poisson-Martin boundary integral
representation theorem for $\la$-harmonic functions. 

Let us write  $\res^*(P)$ for the set of all elements in the resolvent set of $P$
as a self-adjoint operator for which the diagonal matrix elements of the resolvent
($\la$-Green function) do not vanish. For $\la \in \res^*(P)$, the classical  
weights satisfying the needed axioms are suitable quotients of the $\la$-Green function,
which we call the \emph{Green weights.}
This leads to the above mentioned representation proved in \cite{PiWo} and the preceding work.

Later on (\S 4), we restrict attention to the case when $P$ is invariant under
a transitive group of automorphisms of $\,T$. In this situation, we discuss the cases where
in addition to the classical ones, one can find different sets of weights which also
lead to boundary integral representations for the same space of $\la$-harmonic 
functions. In this case, however, we show that the distribution which arises for
a given $\la$-harmonic function does typically not extend to a ($\sigma$-additive) Borel
measure on the boundary, even when this is true with respect to the Green weights.

\section{Boundary and distributions}\label{sec:bdry}

\textbf{A. The end compactification}

\smallskip

For two vertices $x, y \in T$, the \emph{geodesic} or \emph{geodesic path} from $x$ to $y$ is the
unique shortest path $\pi(x,y)$ from $x$ to $y$, and the distance $d(x,y)$ is the
length (number of edges) of $\pi(x,y)$.
 
A \emph{ray} or \emph{geodesic ray} in $T$
is a sequence $[x_0\,,x_1\,,x_2\,,\dots]$ such that $x_{i-1}\sim x_i$ and $x_{i+1}\ne x_{i-1}$
for all $i$. Two rays are \emph{equivalent,} if they differ by finitely many initial
vertices. An \emph{end} of $T$ is an equivalence class of rays. If $x$ is a vertex and $\xi$ an
end, there is a unique geodesic ray $\pi(x,\xi)$ which starts at $x$ and represents $\xi$.  
The boundary $\partial T$ of $T$ is the set of all ends of $T$. 
For $x, y \in T$ with $x \ne y$, the \emph{branch} or \emph{cone} $T_{x,y}$ is the subtree 
spanned by all vertices $w$ with $y \in \pi(x,w)$, and the \emph{boundary arc}
$\partial T_{x,y}$ is the set of all ends which have a representative ray in $T_{x,y}\,$.

We set $\wh T = T \cup \partial T$ and $\wh T_{x,y} = T_{x,y} \cup \partial T_{x,y}\,$.
We put the following topology on $\wh T\,$: it is discrete on the vertex set, 
and a neighbourhood base of $\xi \in \partial T$ is given by the collection of all 
$\wh T_{x,y}$ which contain a ray that represents $\xi$. 
(Here, we may fix $x$ and vary only $y \ne x$.) 
The resulting space is metrizable. It is compact precisely when $T$ is locally finite, but
otherwise, it is not complete. This can be overcome as follows. 
For each vertex  $x$ with infinite degree -- following an idea of 
{\sc Soardi}~\cite{CaSoWo} -- we add a boundary point as follows: we introduce
a new \emph{improper vertex} $x^*$, the \emph{shadow} of $x$, and we set 
$T^* = \{ x^* : x\in T\,,\;\dg(x) = \infty \}$, as well as $\partial^*T = T^* \cup \partial T$.
Analogously, $\partial^* T_{x,y} = T_{x,y}^* \cup \partial T_{x,y}\,.$

Let us write $\ol{T} = T^* \cup \wh T$ and $\ol{T}_{x,y} = T_{x,y}^* \cup \wh T_{x,y}\,.$
Again, $T$ is discrete in $\ol{T}$. A neighbourhood base of end $\xi\in \partial T$ is
now provided by all  $\ol{T}_{x,y}$ which contain a geodesic that represents $\xi$.
A neighbourhood sub-base of $x^* \in T^*$ is given by all $\ol T_{v,x}$, where $v \sim x$.

We now describe convergence of sequences in $\ol{T}$ in this topology. 
We choose a \emph{root} vertex $o \in T$ and write $T_x = T_{o,x}$ for any $x \in T$; 
in particular, $T_o=T$. Throughout everything which follows, it will 
be useful to define the \emph{predecessor} $x^-$ of a vertex $x \ne o$. 
This is the neighbour of $x$ on the geodesic $\pi(o,x)$, and $x$ is a called a 
\emph{forward neighbour} of $x^-$. For $x \in T$, set $|x| = d(o,x)$, the
graph distance. For $\xi \in \partial T$, set $|\xi|=\infty$.

For any pair of elements $v, w \in \wh T$ (i.e., not in $T^*$), 
their \emph{confluent} $v \wedge w$ with respect
to $o$ is the last common element on the geodesics $\pi(o,v)$ and $\pi(o,w)$.
It is a vertex, unless $v=w \in\partial T$, in which case the confluent is
that end. Now, if $(w_n)$ is a sequence in $\wh T$, then
\begin{itemize}
 \item $\; w_n \to x \in T$ when $w_n = x$ for all but finitely $n$.
 \item $\; w_n \to \xi \in \partial T$ when $|w_n \wedge \xi| \to \infty\,.$
 \item $\; w_n \to x^* \in T^*$ when $w_n$ ``rotates'' around $x$, that is,
any $y \sim x$ lies on at most finitely many geodesics $\pi(x,w_n)$. 
\end{itemize}

Finally, if $(x_n^*)$ is a sequence of improper vertices, then
\begin{itemize}
 \item $\; x_n^* \to x^* \in T^*$ or $x_n^* \to \xi \in \partial T$ when
$x_n \to x^*$, resp. $x_n \to \xi$ in the above sense.
\end{itemize}

Now $\ol{T}$ is compact, and $T$ is an open-discrete subset, so that  also $\partial^* T$
is compact. For the understanding of distributions, the next considerations will
be useful. They follow {\sc Cartwright, Soardi and Woess}~\cite{CaSoWo}, see also 
\cite[Thm. 7.13]{W-Markov}.

Let $X$ be a countable set. By a \emph{compactification} of $X$ we mean a compact
Hausdorff space into which $X$ embeds as a discrete, open, dense subset.
Now let $\mathcal{F}$ be a \emph{countable} family of bounded functions 
$f:X \to \R$. Then there is a unique minimal compactification $\ol{X}_{\mathcal{F}}$
of $X$ such that each $f \in \mathcal{F}$ extends to a continuous function on
$\ol{X}_{\mathcal{F}}\,$. Here, ``minimal'' refers to the partial order on compactifications
where one is smaller than the other if the identity mapping on $X$ extends to a continuous
surjection from the bigger to the smaller one, and two compactifications are
considered equal, if that extension is a homeomorphism.

Now let $T$  be a countable tree (or any connected, countable graph) with edge set 
$E = \{ (x,y) \in T^2 :x \sim y\}$.
A function $f: T \to \C$ is called \emph{locally constant,} if the set of edges
along wich $f$ changes its value, 
$$
\{ e = [x,y] \in E : f(x) \ne f(y) \},
$$   
is finite. The vector space $\mathcal{V}$ of all locally constant functions is spanned
by the countable set $\mathcal{F}$ of all those functions in $\mathcal{V}$ which take
values in $\{ 0, 1\}$. Therefore, in the corresponding compactification 
$\ol{T}_{\mathcal{F}}\,$, every locally constant function on $T$ has a continuous extension.
Now, as explained in \cite{CaSoWo}, when the tree (graph) is locally finite,
then one gets the well-known \emph{end compactification.} When the tree $T$ is not locally
finite, we just get the compactification $\ol{T} = T^* \cup \wh T$ described above.

For the purposes of the present note, the improper vertices remain an artifact
which provides compactness, but will not be used in a specific way, except to clarify the view
on the subject.

\medskip

\textbf{B. Distributions on the boundary}

\smallskip

The following material is adapted and extended from \cite{PiWo}.
Consider a function $f \in \mathcal{V}$. Let $E_f$ be the finite 
set of edges along which $f$ changes value. We can choose a finite
subtree $\tau$ of $T$ which contains all those edges as well as
the chosen root $o$. For a vertex $x \in \tau$, write $S_x(\tau)$ 
for the set of forward neighbours $y$ of $x$ in $\tau$ (it may be empty).
The \emph{boundary} $\partial \tau$ of $\tau$ in $T$
consists of those $x \in \tau$ which have a neighbour outside $\tau$.
For each $x \in \partial \tau$, the function $f$ is constant on the 
part of the tree which branches off at $x$, which is
$$
T_x(\tau) = T_x \setminus \bigcup \{ T_y : y \in S_x(\tau) \}
$$
Now let $\partial \mathcal{V}$ be the trace of the vector space $\mathcal{V}$
on $\partial T$, and define $\partial \mathcal{F}$ correspondingly. By the above
considerations, each element of $\partial \mathcal{F}$ is the indicator function of a subset of  
$\partial T$ which can be written as a finite disjoint union of
sets of the form
$$
\partial T_x \setminus \bigcup \{ \partial T_y  : y \in S_x\}\,,
$$
where $x \in T$ and $S_x$ is a finite collection of forward neighbours of $x$ (possibly
empty). 
If $\nu$ is an element in the dual space of $\partial \mathcal{V}$ then it can be seen
as a complex-valued set function on the collection of all those sets, and we call it a \emph{distribution.}
The following is now obvious.

\begin{lem}\label{lem:dist} Any distribution $\nu$ is characterized by the
property that, for every $x \in T$ and finite set $S_x$ of forward neighbours of $x$,
$$
\nu(\partial T_x) = \sum_{y \in S_x} \nu(\partial T_y) + 
\nu\Bigl(\partial T_x \setminus \bigcup \{ \partial T_y  : y \in S_x\}\Bigr)\,.
$$
\end{lem}

In particular, if $T$ is locally finite, then $\nu$ can be described as
a set function on all boundary arcs such that 
\begin{equation}\label{eq:nu}
\nu(\partial T_x) = \sum_{y: y^-=x} \nu(\partial T_y) \; \text{ for every }\; x \in T.
\end{equation}
In \cite{PiWo}, we have defined distributions analogously in the non-locally finite
case, requiring in that case that the sum in \eqref{eq:nu} converges absolutely.
In this case, let us call $\nu$ a \emph{strong distribution} here. For all results 
of \cite{PiWo} as well as the present note, the 
distributions actually involved are always strong. 

In particular, when $\nu$ is non-negative real, then it is not only strong,
but extends to a finite, $\sigma$-additive Borel measure on $\partial T$, as explained
in \cite[3.10]{PiWo}. As mentioned there, when $\nu$ is a complex-valued distribution,
a necessary and sufficient condition for its extendability to a $\sigma$-additive, signed
measure on the Borel $\sigma$-algebra of $\partial T$ is that 
there is $M < \infty$ such that for any sequence
of pairwise disjoint boundary arcs $\partial T_{x_n}\,$, one has 
\begin{equation}\label{eq:M}
\sum_n |\nu(\partial T_{x_n})| \le M \,.
\end{equation}
This is an easy extension of the corresponding condition in the
locally finite case of {\sc Cohen, Colonna and Singman}~\cite{CoCoSi}.

For any distribution $\nu$, we now write 
$$
\nu(\varphi) = \int_{\partial T} \varphi\, d\nu\quad \text{for}\quad \varphi \in \partial \mathcal{V}\,.
$$
By the above, given $\varphi$, there are a finite subtree $\tau$ of $T$ containing $o$
and constants $\varphi_x\,$, $x \in \partial \tau$, such that
\begin{equation}\label{eq:intnu}
\begin{gathered}
\varphi \equiv \varphi_x \quad \text{on} \quad 
\partial T_x(\tau) = \partial T_x \setminus \bigcup \{ \partial T_y  : y \in S_x(\tau) \}\,,\AND\\
\int_{\partial T} \varphi\, d\nu = \sum_{x \in \partial \tau} \varphi_x\, 
  \nu\bigl(\partial T_x(\tau)\bigr)\,.
\end{gathered}
\end{equation}
By construction, this does not depend on the specific choice of the finite tree $\tau$
associated with $\varphi$.
If $\nu$ extends to a $\sigma$-additive complex Borel measure on $\partial T$, then the
integral is an ordinary one in the sense of Lebesgue.

\medskip

\textbf{C. Self-adjointness of the transition operator}

\smallskip

With the action defined by \eqref{eq:Pf}, the transition operator $P$ is
self-adjoint on the Hilbert space
$\ell^2(T,\mm)$ of all functions $f:T \to \C$
with $\langle f, f \rangle < \infty$, where
$$
\langle f, g \rangle = \sum_{x} f(x)\ol{g(x)}\,\mm(x)\,,
$$ 
with the measure $\mm$ on $T$ as follows: 
$$
\text{for }\; x \in T\;\text{ with }\; \pi(o,x) = [x_0\,,x_1\,,\dots, x_k]\,,\quad
\mm(x)= \frac{p(x_0\,,x_1) \cdots p(x_{k-1}\,,x_k)}{p(x_1\,,x_0) \cdots p(x_k\,,x_{k-1})}.
$$
In particular, $\mm(o)=1$. Self-adjointness is a consequence of  \emph{reversibility:} 
$\mm(x)p(x,y) = \mm(y)p(y,x)$ for all $x,y$.
The norm (spectral radius) of $P$ is 
$$
\rho = \rho(P) = \limsup_{n \to \infty} p^{(n)}(x,y)^{1/n}
$$
(independent of $x$ and $y$), where $p^{(n)}(x,y)$ is the $(x,y)$-element 
of the matrix power $P^n$. Since trees are bipartite, the spectrum
$\spec(P) \subset [-\rho\,,\,\rho]$ is symmetric around the origin.

\emph{Positive} $\la$-harmonic functions exist if and only if $\la \ge \rho$ (real).
At this point, we state a warning: when viewing $\la$-harmonic functions as
``eigenfunctions'' of $P$,  they are not considered as eigenfunctions of
the above self-adjoint operator on $\ell^2(T,\mm)$. As a matter of fact, besides
possibly for $\la = \pm \rho$ in very specific situations, our $\la$-harmonic functions 
will usually not belong to
$\ell^2(T,\mm)$. In a variety of known cases, $\spec(P)$ contains no eigenvalues, that is, there
is no point spectrum on $\ell^2(T,\mm)$. In any case, our methods and results do not 
cover the case where $\la \in \spec(P) \setminus\{\pm \rho \}$.

\section{The general integral representation}\label{sec:intrep}

We now fix a candidate eigenvalue $\la \in \C$ and we suppose that we can equip
the \emph{oriented} edge set $E(T) = \{ (x,y) \in T^2 : x \sim y\}$ of $T$  with
\emph{$\la$-weights} $f(x,y) \in \C$ satisfying the following properties
for every $x \in T$ and every $y$ with $x \sim y$.

\begin{gather}
f(x,y)f(y,x) \ne 1 \quad 
\text{for all pairs of neighbours }\; x,y\,, \label{eq:ne1}\\
u(x,x) \ne \la\,,\quad\text{where}\quad u(x,x) = \sum_v p(x,v)f(v,x)\,,\label{eq:nela}\\
\la \, f(x,y) = p(x,y) + \Bigl(u(x,x)-p(x,y)f(y,x)\Bigr)f(x,y)\,. \label{eq:la}
\end{gather}

If $\dg(x) = \infty$ then we require that the sum in \eqref{eq:nela}
converges absolutely. Note that it follows from \eqref{eq:la} that $f(x,y) \ne 0$ for all
pairs of neighbours. The above three axioms arise by mimicking the main properties of
the natural Green weights, which will be discussed at the end of this section.

Using these weights,  for arbitrary $x, y \in T$ we define
\begin{equation}\label{eq:f}
f(x,y) = f(x_0\,,x_1) f(x_1\,,x_2)\cdots f(x_{k-1}\,,x_k)\,, \quad \text{if}\quad
\pi(x,y) = [x_0\,,\dots, x_k]\,.
\end{equation}
In particular, $f(x,x)=1$. 

\begin{lem}\label{lem:Pf}
For fixed $y$, the function $x \mapsto f(x,y)$ satisfies
$$
Pf(x,y) = \la \, f(x,y) \quad \text{if}\;x \ne y\,,\AND
Pf(y,y) = u(y,y)\,.
$$   
\end{lem} 

\begin{proof}
The second identity is the definition
\eqref{eq:nela} of $u$. For the first identity, let $\pi(x,y)$ be as in \eqref{eq:f}. Consider the neighbours $x=x_0$ and $x_1$. Then
 \eqref{eq:nela} and \eqref{eq:la} yield
\[
\begin{aligned} 
Pf(x,y) &= p(x,x_1) f(x_1\,,y) + \sum_{v \ne x_1} p(x,v) f(v,x) f(x,y) \\
&= p(x,x_1)f(x_1\,,y) + \Bigl( u(x,x) - p(x,x_1)f(x_1,x) \Bigr) f(x,x_1)f(x_1,y)\\
& = \la\, f(x,x_1) f(x_1,y) = \la \, f(x,y)\,,
\end{aligned}
\]
as stated.
\end{proof}

Note that absolute convergence of the sum in \eqref{eq:nela} is crucial for the Lemma.
It is this property that further on will give us strong distributions. Thanks to \eqref{eq:nela} we can
set 
\begin{equation}\label{eq:ug}
g(x,y) = \frac{f(x,y)}{\la - u(y,y)}\,,\quad x,y \in T\,.
\end{equation}
Then we see from Lemma \ref{lem:Pf} that the function $x \mapsto g(x,y)$ satisfies the resolvent equation
\begin{equation}\label{eq:Pg}
Pg(x,y) = \la\, g(x,y) - \de_x(y)\,.
\end{equation}

The following Lemma shows how the transition probabilities 
can be recovered
from the weights $f(x,y)$, compare with \cite{Ca} for the locally finite case with 
standard non-negative Green weights.

\begin{lem}\label{lem:recover} For $x \in T$ and $y\sim x$,
$$
\begin{aligned}
g(x,x)p(x,y) &= \frac{f(x,y)}{1-f(x,y)f(y,x)}\,,\\[3pt]
g(x,x)g(y,y) &= g(x,y)\biggl(\frac{1}{p(x,y)} + g(y,x)\biggr), \quad\text{and}\\[3pt]
\la\, g(x,x) &= 1 + \sum_{y: \,y \sim x}   \frac{f(x,y)f(y,x)}{1-f(x,y)f(y,x)}\,.
\end{aligned}
$$ 
When $\dg(x)=\infty$, the last sum converges absolutely.
\end{lem}

\begin{proof}
We can rewrite \eqref{eq:la} as
$$
p(x,y) \bigl(1 - f(x,y)f(y,x) \bigr) = f(x,y) \bigl( \la - u(x,x) \bigr)\,.
$$
Since $\la - u(x,x) = 1/g(x,x)$, the first identity follows, and
with $g(x,y) = f(x,y)g(y,y)$ as well as $g(y,x) = f(y,x)g(x,x)$, we get
$$
g(x,y)\Bigl(\frac{1}{p(x,y)} + g(y,x)\Bigr) = g(x,x)g(y,y)\, 
\underbrace{f(x,y)\Bigl(\frac{1}{g(x,x)p(x,y)} + f(y,x)\Bigr)}_{\displaystyle = 1}.   
$$ 
This is the second identity. We now multiply the first identity  with $f(y,x)$. 
The sum over all neighbours $y$ of $x$ is absolutely convergent by assumption, so that
we have indeed abolute convergence of the right hand side of the third identity,
and 
$$
\sum_{y: y \sim x} \frac{f(x,y)f(y,x)}{1-f(x,y)f(y,x)} = g(x,x)\,u(x,x) 
= \la\, g(x,x) - 1
$$ 
by the definition of $g(x,x)$. 
\end{proof}

We define the \emph{$\la$-Poisson kernel} associated with our weights by
$$
k(x,w) = \frac{f(x,x \wedge w)}{f(o,x \wedge w)} \,,\quad x \in T\,,\; w \in \wh T\,.
$$
Thus,
\begin{equation}\label{eq:kern}
k(x,w) = \frac{f(x,v)}{f(o,v)} = \frac{g(x,v)}{g(o,v)} \quad \text{for every vertex }\;
v \in \pi(x \wedge w,w).
\end{equation}
By our assumptions, $Pk(\cdot,w)$ is well defined as a function of the first variable.
That is, even at vertices with infinite degree, the involved sum is absolutely convergent,
and for $\xi \in \partial T$,
\begin{equation}\label{eq:Pk}
\sum_{y \sim x} p(x,y) \,k(y,\xi)= \la\, k(x,\xi) \quad \text{for every}\; x \in T.
\end{equation}
Now let $x \in T$ and $\pi(o,x) = [o=x_0\,,x_1\,,\dots, x_k=x]$.
Then $x \wedge \xi \in \{ x_0\,,x_1\,,\dots, x_k\}$ for every $\xi \in \partial T$,
and 
$$
k(x,\xi) = k(x,x_i) \quad\text{when}\quad 
\begin{cases} \xi \in \partial T_{x_i}\setminus \partial T_{x_{i+1}}\,,&i \le k-1,\\
              \xi \in \partial T_{x_k}\,,&i=k\,.
\end{cases}
$$  
Thus, $\varphi = k(x,\cdot)$ is locally constant on $\partial T$, and 
we can use $\pi(o,x)$ as a tree $\tau$ to which \eqref{eq:intnu} applies. 
Then we have the following.

\begin{pro}\label{pro:transf}
If $\nu$ is a strong distribution $\nu$ on $\partial T$, its \emph{Poisson transform} 
$$
h(x) = \int_{\partial T} k(x,\xi)\, d\nu(\xi)
$$
is a $\la$-harmonic function, and
\begin{equation}\label{eq:int'}
\begin{aligned}
h(x) 
&= \sum_{i=0}^{k-1} k(x,x_i) 
\Bigl(\nu(\partial T_{x_i}) - \nu(\partial T_{x_{i+1}})\Bigr) + k(x,x)\,\nu(\partial T_{x})\\
&= k(x,o)\nu(\partial T) + \sum_{i=1}^k \Bigl(k(x,x_i)-   k(x,x_{i-1})\Bigr) \nu(\partial T_{x_i})\,.
\end{aligned}
\end{equation}
\end{pro}

\begin{proof} The proof of $\la$-harmonicity of $h$ is obvious when 
$T$ is locally finite. Otherwise, some care is needed, and we go through the details 
in order to show the necessity of the assumption that the distribution $\nu$ 
be strong. First of all,  we 
show that $Ph(o) = \la\,h(o)$. 
By \eqref{eq:int'}, if $x \sim o$, 
\begin{equation}\label{eq:Pho}
h(x) = f(x,o)\nu(\partial T) + \left(\frac{1}{f(o,x)} - f(x,o)\right)\nu(\partial T_x).
\end{equation}
Therefore 
$$
\begin{aligned}
Ph(o) &= u(o,o)\, \nu(\partial T) + \sum_{x\sim o}
p(o,x) \;\frac{1-f(o,x)f(x,o)}{f(o,x)} \,\nu(\partial T_x)\\
&=  \sum_{x\sim o} \left(u(o,o) + p(o,x) \;\frac{1-f(o,x)f(x,o)}{f(o,x)}\right) \nu(\partial T_x)\\
&= \la \sum_{x\sim o} \nu(\partial T_x) = \la\, \nu(\partial T) = \la\, h(o)
\end{aligned}
$$
by \eqref{eq:la}. In case $\dg(x) = \infty$, we needed absolute convergence of the
involved series. Similarly, let $x \ne o$. Then \eqref{eq:int'} yields the formula
\begin{equation}\label{eq:recurs}
 h(x) = f(x,x^-)h(x^-) + \frac{1-f(x,x^-)f(x^-,x)}{f(o,x)}\;\nu(\partial T_x)\,,
\end{equation}
that will also be important further below. To prove \eqref{eq:recurs}, we first observe that it
is the same as \eqref{eq:Pho} when $x \sim o$. Now let $k \ge 2$ in \eqref{eq:int'}, and note
that for $i \le k-1$ we have $k(x,x_i) = f(x,x_i)/f(o,x_i)= f(x,x^-)k(x^-,x_i)$, 
with $x^- = x_{k-1}\,$. Then, using the first of the two identities of \eqref{eq:int'},
$$
\begin{aligned}
h(x) &= f(x,x^-) \sum_{i=0}^{k-2} k(x^-,x_i) 
\Bigl(\nu(\partial T_{x_i}) - \nu(\partial T_{x_{i+1}})\Bigr)\\ &\qquad + 
f(x,x^-) k(x^-,x^-) \Bigl(\nu(\partial T_{x^-}) - \nu(\partial T_x)\Bigr)
+ k(x,x)\,\nu(\partial T_{x})\\
&= f(x,x^-) h(x^-) - \frac{f(x,x^-)}{f(o,x^-)}\nu(\partial T_x) + \frac{1}{f(o,x)}\nu(\partial T_x)\,.
\end{aligned}
$$
Since $f(o,x)=f(o,x^-)f(x^-,x)$, this reduces to the desired formula.
For the following, we also need \eqref{eq:recurs} for $y$ with $y^-=x$. Absolute convergence
in the first of the following identities is justified a posteriori, and the first
identity of Lemma \ref{lem:recover} is used for the underbraced as well as for the
overbraced term, and again in the very last step. 
$$
\begin{aligned}
Ph(x) &= p(x,x^-)h(x^-) + \sum_{y^-=x} p(x,y)f(y,x)h(x) \\
& \qquad\qquad\qquad + \frac{1}{f(o,x)} \sum_{y^-=x} \underbrace{p(x,y)\frac{1-f(x,y)f(y,x)}{f(x,y)}}_{
\displaystyle 1/g(x,x)} \nu(\partial T_y)\\[-4pt]
&= \frac{p(x,x^-)}{f(x,x^-)}h(x) 
- \frac{1}{f(o,x)}\overbrace{p(x,x^-)\frac{1-f(x,x^-)f(x^-,x)}{f(x,x^-)}}^{} \nu(\partial T_x)\\
& \qquad\qquad\qquad +
u(x,x)h(x) - p(x,x^-)f(x^-,x)h(x) + \frac{1}{g(o,x)} \nu(\partial T_x)\\
&= \left( p(x,x^-) \frac{1-f(x,x^-)f(x^-,x)}{f(x,x^-)} + u(x,x)\right) h(x) = \la\, h(x)\,.
\end{aligned} 
$$
In the second identity we made use of the assumption that $\nu$ is strong.
\end{proof}

The proof of the following is very similar to \cite[Thm. 9.37]{W-Markov}: 
we rewrite its main part here
to take care of absolute convergence in the non-locally finite case.

\begin{thm}\label{thm:general}
Suppose that we have edge weights $f(x,y)$ which satisfy \eqref{eq:ne1} -- \eqref{eq:la}.
A function $h: T \to \C$ satisfies $Ph = \la \cdot h$ if and only
if it is of the form
$$
h(x) = \int_{\partial T} k(x,\xi)\,d\nu(\xi)\,,
$$
where $\nu$ is a strong complex distribution on $\partial T\,.$ The distribution $\nu$ is 
determined by $h$, that is, $\nu = \nu^h$, where  
$$
\nu^h(\partial T)= h(o) \AND 
\nu^h(\partial T_x) = f(o,x)\,\frac{h(x)-f(x,x^-)h(x^-)}{1-f(x^-,x)f(x,x^-)}\,,
\; x \ne o\,.
$$
\end{thm}

\begin{proof} We first show that if $h$ is $\la$-harmonic, then $\nu^h$ as defined in the 
theorem is indeed a strong distribution, and $h$ is its Poisson transform.
We start with the identity
$$
\la\,g(x,x)h(x) = \sum_{y: y \sim x}  g(x,x) p(x,y) h(y)\,,
$$ 
and recall that the sum on the right hand side is assumed to converge absolutely  when $\dg(x)=\infty$.
Using Lemma \ref{lem:recover}, we rewrite this as
$$ 
\biggl( 1 + \sum_{y: y \sim x}   \frac{f(x,y)f(y,x)}{1-f(x,y)f(y,x)}\biggr) h(x) =
\sum_{y: y \sim x} \frac{f(x,y)   
}{1-f(x,y)f(y,x)}\, h(y)\,.
$$
Since the involved sums converge absolutely, we can regroup the terms and
get 
\begin{equation}\label{eq:hidentity}
h(x) = \sum_{y\,:\,y \sim x} f(x,y)\;\frac{h(y)-f(y,x)h(x)}{1-f(x,y)f(y,x)}\,.
\end{equation}
Convergence is again absolute when $\dg(x)=\infty\,$. 

For $x = o$, the last identity says that
$\nu^h(\partial T) = \sum_{y \sim o} \nu^h(\partial T_y)$.
If $x \ne o$, then by \eqref{eq:hidentity}, 
$$
\begin{aligned}
\sum_{y: y^- = x} \nu^h(\partial T_y) 
&= f(o,x) \sum_{y\,:\,y^- = x} f(x,y) \frac{h(y)-f(y,x)h(x)}{1-f(x,y)f(y,x)} \\
&= f(o,x)
\left(h(x) -f(x,x^-)\frac{h(x^-)-f(x^-,x)h(x)}{1-f(x,x^-)f(x^-,x)}\right)\\
&= f(o,x)\frac{h(x)-f(x,x^-)h(x^-)}{1-f(x,x^-)f(x^-,x)}
=\nu^h(\partial T_x).
\end{aligned}
$$
So $\nu^h$ is indeed a signed distribution on $\mathcal{F}_o\,.$
We verify that $\int_{\partial T} k(x,\xi)\,d\nu^h(\xi) = h(x)$.
For $x = o$ this is clear, so let $x \ne o$. With  
notation as in \eqref{eq:int'}, we simplify
$$
\Bigl( k(x,x_i) - k(x,x_{i-1})\Bigr)\,\nu^h(\partial T_{x_i})
= f(x,x_i)h(x_i)-f(x,x_{i-1}) h(x_{i-1})\,,
$$
whence we obtain
$$
\begin{aligned}
\int_{\partial T} K(x,\xi)\,d\nu^h(\xi) &= k(x,o)h(o) +
\sum_{i=1}^k \Bigl(f(x,x_i)h(x_j)-f(x,x_{i-1}) h(x_{i-1})\Bigr)\\ 
&= f(x,x)h(x) = h(x)\,,
\end{aligned}
$$
as stated.

Second, we need to verify that given $\nu$ and its Poisson transform 
$h$, we have $\nu=\nu^h$. This part of the proof is nothing but the identity 
\eqref{eq:recurs} in the proof of Proposition \ref{pro:transf}.
\end{proof}

\smallskip

\textbf{The natural Green weights}

\smallskip

We now ``reveal'' the origin of the axioms \eqref{eq:ne1} -- \eqref{eq:la} for the edge 
weights. Let $\res(P)$ be the resolvent set of the self-adjoint operator $P$ acting on
$\ell^2(T,\mm)$ according to \S \ref{sec:bdry}.C. For $\la \in \res(P)$, we write 
$\mathfrak{G}(\la) = (\la\cdot I - P)^{-1}$ for resolvent operator. Its matrix element
\begin{equation}\label{eq:Gxy}
G(x,y|\la) = \mathfrak{G}(\la)\uno_y(x)\,,\quad x,y \in T,
\end{equation}
is the \emph{Green function.} It is an analytic function of 
$\la \in \res(P) \supset \C \setminus [-\rho\,,\,\rho]$, and for $|\la| > \rho$,
$$
G(x,y|\la) = \sum_{n=0}^{\infty} p^{(n)}(x,y)\,\la^{-n-1}\,. 
$$
At $\la = \rho$, the latter series converge or diverge simultaneously for all $x, y$.
If they converge, i.e., $G(x,y|\rho) < \infty$ for all $x,y$, then $P$, resp.~the associated
random walk, is called \emph{$\rho$-transient,} and otherwise it is called
\emph{$\rho$-recurrent.} Set
\begin{equation}\label{eq:res*}
\res^*(P) = \bigl\{ \la \in \res(P) : G(x,x|\la) \ne 0 \; \text{for all} \; x \in T \bigr\}.
\end{equation}
For $\la \in \res^*(P)$, 
\begin{equation}\label{eq:Fxy}
F(x,y|\la) = G(x,y|\la)/G(y,y|\la)\,,\quad x,y \in T\,,
\end{equation}
is an analytic function of $\la$. For $|\la| \ge \rho$, 
\begin{equation}\label{eq:FU}
F(x,y|\la) = \sum_{n=0}^{\infty} f^{(n)}(x,y)/\la^n\,,
\end{equation}
where $f^{(n)}(x,y)$ is the probability that the random walk
starting at $x$ hits $y$ at time $n \ge 0$ for the first time. Also,
$$
U(x,x|\la) = \sum_{y \sim x} p(x,y) F(y,x|\la) = \sum_{n=1}^{\infty} u^{(n)}(x,x)/\la^n\,,
$$
where $u^{(n)}(x,x)$ is the probability that the random walk
starting at $x$ returns to $x$ at time $n \ge 1$ for the first time.
Now it is well known, and also explained in \cite{Ca}, \cite{W-Markov} as well as
in \cite{PiWo}, that the edge weights
$$
f(x,y) = F(x,y|\la)\,, \quad x,y \in T,\; x \sim y
$$
are $\la$-weights which fulfill the requirements \eqref{eq:ne1} -- \eqref{eq:la} 
for $\la \in \res^*(P)$, and for arbitrary $x,y \in T$,
\begin{equation}\label{eq:F}
\begin{aligned}
F(x,y|\la) &= F(x_0\,,x_1|\la) \cdots F(x_{k-1}\,,x_k|\la)\,,\quad \text{ where }\\ 
\pi(x,y) &= [x=x_0\,,x_1\,,\dots, x_k=y]. 
\end{aligned}
\end{equation}
With notation as in \S \ref{sec:intrep}, we also have $u(x,x) = U(x,x|\la)$ and 
$g(x,y) = G(x,y|\la)$. 
The associated kernel according to \eqref{eq:kern}, called the \emph{$\la$-Martin kernel}, is
\begin{equation}\label{eq:Mkernel}
k(x,w) = K(x,w|\la) = \frac{G(x,v|\la)}{G(o,v|\la)} \quad \text{for every vertex }\; 
v \in \pi(x \wedge w, w),
\end{equation}
where $x \in T$ and $w \in \wh T$. All this also works for $\la = \pm \rho$ in the
$\rho$-transient case. Thus, Theorem \ref{thm:general} yields the following, which we restate
here once again.

\begin{cor}\label{cor:martin}
For $\la \in \res^*(P)$, as well as for $\la = \pm \rho$ in the $\rho$-transient case,
every $\la$-harmonic function $h$ has an integral representation
$$
h(x) = \int_{\partial T} K(x,\xi|\la)\,d\nu(\xi)\,.
$$
The strong complex distribution $\nu = \nu^h$ on $\partial T$ is 
determined by $h$, 
$$
\nu^h(\partial T)= h(o) \AND 
\nu^h(\partial T_x) = F(o,x|\la)\,\frac{h(x)-F(x,x^-|\la)h(x^-)}{1-F(x^-,x|\la)F(x,x^-|\la)}\;\;\;
\text{for}\; x \ne o\,.
$$
\end{cor}

As already mentioned, this general result of \cite{PiWo} was preceded by various earlier 
ones, starting with the seminal paper \cite{Ca} (that deals with locally finite trees and  
positive $\la > \rho$,  and also $\la = \rho$ in the $\rho$-transient case), and another 
proof in \cite{PiWo-0}. In \cite{FiSt}, one finds the result for complex $\la$ in the locally 
finite case corresponding to nearest neighbour group invariant random walks on free groups 
(resp. closley related groups freely generated by involutions): the special case 
of the simple random walk in this environment goes back to \cite{MaZa}.
A first proof for the non-locally finite case and $\la=1$ (transient case) is in 
\cite[\S 9.D]{W-Markov}.  

\begin{rmk}\label{rmk:positive}
If $\la \ge \rho$, or if $\la=\rho$ in the $\rho$-transient case, it is
a well-known fact that for any positive $\la$-harmonic function $h$, one has
$$
F(x,y|\la)\, h(y) \le h(x)\quad \text{for all }\; x, y\,.
$$
(This holds for any irreducible Markov chain.) In particular, the distribution $\nu^h$ of 
Corollary \ref{cor:martin} is non-negative, whence it extends to a $\sigma$-additive measure
on $\partial T$,  and Corollary \ref{cor:martin} leads to the classical Poisson-Martin 
representation. 
Furthermore, in that case, the real $\la$-harmonic functions which are Poisson transforms
of \emph{$\sigma$-additive} signed Borel measures on $\partial T$ are precisely the differences
of non-negative $\la$-harmonic functions. For the complex-valued case, the situation is
analogous. \hfill $\square$
\end{rmk}

There are many analogies between the structure, group actions, harmonic analysis and
potential theory on trees (in particular, regular trees) and the Poincar\'e disk, that
is, the open unit disk with the hyperbolic metric. The discrete Laplacian $P-I$ arising
from a random walk on $T$ is an analogue of the hyperbolic Laplace-Beltrami operator on the
disk. See e.g. {\sc Boiko and Woess}~\cite{BoWo} for a mostly potential theoretic
``dictionary'' regarding the correspondences. In this sense, our representation theorem
\ref{thm:general} should be seen as a discrete analogue of a result of 
{\sc Helgason}~\cite{He} for a Poisson-type integral representation of \emph{all} harmonic
functions on rank 1 symmetric spaces, and in particular, the hyperbolic disk: see the
beautifully written exposition by {\sc Eymard}~\cite{Ey}. There, the integral representation
is with respect to \emph{analytic functionals} on the boundary (the unit circle), of which
our strong distributions are the analogues in the tree setting.

\section{Twin kernels for affine  and simple random walks}\label{sec:twin}

As we have seen above, the natural version of Theorem \ref{thm:general} 
is the one where the $\la$-weights are $f(x,y) = F(x,y|\la)$, where $\la \in \res^*(P)$,
resp.~$\la = \pm \rho$ in the $\rho$-transient case.

Now, there are cases where one has another 
choice for the collection of $\la$-weights $f(x,y)$ satisfying
\eqref{eq:ne1} -- \eqref{eq:la}, leading to another kernel which can also be used
to describe the $\la$-harmonic functions of $P$. The main aim of this section is to obtain a
better understanding of such twin kernels and the different integral representations
for a class of random walks which includes the simple random walk on a homogeneous tree.

We consider $T = \T_q\,$, the homogeneous tree with degree $q+1$, where $q \ge 1$. 
In case $q=1$, this is just the bi-infinite integer line $\Z$.

For any end $\xi$ of $T$, we define the associated \emph{horocycle index}
$$
\hor(x,\xi) = d(x,x \wedge \xi) - d(o,x \wedge \xi) \in \Z\,,
$$
(we recall that $\wedge$ stands for taking the confluent with respect to $o$).
In addition to the root vertex, we choose and fix a reference end $\varpi$
and write $\hor(x) = \hor(x, \varpi)\,$. 
The \emph{horocycles} are the resulting level sets:
$\Hor_k = \{ x \in T : \hor(x) = k \}\,$, $k \in \Z$. Thus, (following Cartier) 
one can imagine the tree as an infinite genealogical tree, where $\varpi$ is the 
mythical ancestor, and the horocycles are the successive generations. Each of them 
is infinite, and each $x \in \Hor_k$ has precisely one neighbour (parent) in $\Hor_{k-1}$
and $q$ neighbours (children) in $\Hor_{k+1}$ (see  Figure~1). 
The subgroup of $\Aut(\T_q)$ which preserves
this genealogical order, i.e., the group of automorphisms which fix $\varpi$, is called
the affine group $\Aff(\T_q)$ of $\T_q$. It was shown to be amenable by {\sc Nebbia}~\cite{Ne},
but non-unimodular for $q \ge 2$, see {\sc Trofimov}~\cite{Tr}. We note that the indexing of the
horocycles here is opposite to the one which is commonly used in the unit disk, resp.~hyperbolic
upper half plane. The reason lies in the opposite behaviour of absolute values and $q$-adic
norms.  Very general random walks on $\Aff(\T_q)$ were studied in detail by
{\sc Cartwright, Kaimanovich and Woess}~\cite{CaKaWo}. 
\begin{figure}[ht]
$$
\beginpicture 

\setcoordinatesystem units <.55mm,.9mm> 

\setplotarea x from -10 to 104, y from -12 to 80

\arrow <6pt> [.2,.67] from 2 2 to 80 80

\arrow <6pt> [.2,.67] from 94 2 to 99 -3

 \plot 32 32 62 2 /

 \plot 16 16 30 2 /

 \plot 48 16 34 2 /

 \plot 8 8 14 2 /

 \plot 24 8 18 2 /

 \plot 40 8 46 2 /

 \plot 56 8 50 2 /

 \plot 4 4 6 2 /

 \plot 12 4 10 2 /

 \plot 20 4 22 2 /

 \plot 28 4 26 2 /

 \plot 36 4 38 2 /

 \plot 44 4 42 2 /

 \plot 52 4 54 2 /

 \plot 60 4 58 2 /

 \plot 99 29 64 64 /

 \plot 66 2 96 32 /

 \plot 70 2 68 4 /

 \plot 74 2 76 4 /

 \plot 78 2 72 8 /

 \plot 82 2 88 8 /

 \plot 86 2 84 4 /

 \plot 90 2 92 4 /

 \plot 94 2 80 16 /


\setdots <3pt>
\putrule from -4.8 4 to 102 4
\putrule from -4.1 8 to 102 8
\putrule from -2 16 to 102 16
\putrule from -1.7 32 to 102 32
\putrule from -1.7 64 to 102 64

\put {$\vdots$} at 32 -3
\put {$\vdots$} at 64 -3

\put {$\dots$} [l] at 103 6
\put {$\dots$} [l] at 103 48

\put {$\Hor_{-3}$} [l] at -20 64
\put {$\Hor_{-2}$} [l] at -20 32
\put {$\Hor_{-1}$} [l] at -20 16
\put {$\Hor_0$} [l] at -20 8
\put {$\Hor_1$} [l] at -20 4
\put {$\vdots$} at -10 -3
\put {$\vdots$} [B] at -10 70
\put {$\varpi$} at 82 82
\put {$\xi$} at 100 -5.5

\put {$o$} at 7 10.5
\put {$x$} at 70.8 10.5
\put {$v = \varpi \!\wedge\! \xi$} [r] at 63 66.5
\put {$c = x \!\wedge\! \xi$} [r] at 79.6 18.6

\put {$\scriptstyle \bullet$} at 8 8
\put {$\scriptstyle \bullet$} at 72 8
\put {$\scriptstyle \bullet$} at 80 16
\put {$\scriptstyle \bullet$} at 64 64

\put {Figure~1} at 48 -14
\endpicture
$$
\end{figure}

Here we only consider nearest neighbour
walks which are invariant under that group. Their transition probabilities are parametrized
by an $\alpha \in (0\,,\,1)$ as follows:
\begin{equation}\label{eq:alpha}
\text{for }\; x \sim y\,,\quad p(x,y) = 
\begin{cases} \alpha/q\,,&\text{if }\; \hor(y) = \hor(x)+1\,,\\
              1- \alpha\,,&\text{if }\; \hor(y) = \hor(x)-1\,.
\end{cases}  
\end{equation}
The simple random walk arises when $\alpha = q/(q+1)$.
It is easy to see, and a consequence of the next computations, that the spectral radius
is
$$
\rho = \rho(P) = 2\sqrt{\al(1-\al}.
$$
\begin{rmk}\label{obs:ne1}
In the group invariant case, $G(\la) = G(x,x|\la)$ is independent of $x$.
(Do not confuse this with the resolvent operator $\mathfrak{G}(\la)$, of
which $G(\la) = \mathfrak{G}(\la)\uno_x(x)$ is the diagonal matrix element.) 
In the present example, we can use the argument at the end of \cite[Remark 2.8]{PiWo} to see that
 $G(\la) \ne 0$ for any $\la \in \res(P)$. Indeed, as stated there, if $G(\la)=0$ 
then some and thus every $x \in T$ would have a \emph{unique} neighbour $y$
such that $p(x,y)\,G(y,x|\la) = p(y,x)\,G(x,y|\la) = -1$.
But since $\Aff(\T_q)$ acts transitively on the edges (preserving orientation, hence the ``parent relation''),
this would hold for all pairs of neighbours, a contradiction. \hfill $\square$
\end{rmk}
We shall of course see this via explicit computation in a moment.
By group-invariance, there are only two types of functions $F(x,y|\la)$ for
neighbours $x,y$. We set 
$F_+(\la) = F(x,y|\la)$ when $\hor(y) = \hor(x)+1$, and 
$F_-(\la) = F(x,y|\la)$ when $\hor(y) = \hor(x)-1$. As we have mentioned in \S \ref{sec:intrep}.C,
these functions, as $\la$-weights on the edges, satisfy \eqref{eq:ne1} -- \eqref{eq:la}:
see \cite[Lemma 2.3]{PiWo}. A priori, this 
is true for $|\la| > \rho$, and for other $\la \in \spec(P)$, 
one uses analytic continuation. Now \eqref{eq:la}
yields the following equations.
\begin{align}
\la\, F_-(\la) &= (1-\alpha) + \alpha\, F_-(\la)^2\,,\quad\text{and}\label{eq:eq1}\\
\la\, F_+(\la) &= \frac{\alpha}{q} + (1-\alpha) F_+(\la)^2 + \frac{q-1}{q}\,\alpha\, F_-(\la)F_+(\la)\,.
\label{eq:eq2}
\end{align}
Throughout this paper we make the following habitual choice.

\begin{con}\label{con:vention}
Our usual choice for the analytic continuation of the 
square root is the one on the slit plane without the negative half-axis, that is
$\sqrt{r\,e^{i\theta}} = \sqrt{r} \, e^{i\theta/2}$ for $r > 0$ and $-\pi < \theta < \pi$. 
\end{con}

With this in mind, equation \eqref{eq:eq1} has the two solutions
\begin{equation}\label{eq:FtildeF}
F_-(\la) = \frac{\la}{2\alpha}\Bigl( 1 - \sqrt{1 - 4\alpha(1-\alpha)/\la^2}\Bigr),
\quad
\wt F_-(\la) = \frac{\la}{2\alpha}\Bigl( 1 + \sqrt{1 - 4\alpha(1-\alpha)/\la^2}\Bigr).
\end{equation}

The solution that gives rise to the function defined in \eqref{eq:FU}, and thus is associated with
the resolvent $G(x,y|\la)$, is given by 
the convergent series \eqref{eq:FU} in powers of $1/\la$. It must be analytic for 
$\la \in \C \setminus [-\rho\,,\,\rho]$ and decreasing for real 
$\la > \rho$, so it is the former in \eqref{eq:FtildeF}.
If we insert it into \eqref{eq:eq2} then we get 
once more two solutions,
\begin{equation}\label{eq:sol}
F_+(\la) = \frac{\alpha}{(1-\alpha)q} \, F_-(\la) \AND
\wh F_+(\la) = \frac{\alpha}{(1-\alpha)} \, \wt F_-(\la)
\end{equation}
Again, the solution corresponding to the resolvent is $F_+(\la)$.
On the other hand, if in 
\eqref{eq:eq2} we insert $\wt F_-(\la)$ instead of $F_-(\la)$, then we get the following two other solutions of that equation:
\begin{equation}\label{eq:soltilde}
\wt F_+(\la) = \frac{\alpha}{(1-\alpha)q} \, \wt F_-(\la) \AND
\wh{\wt F}_+(\la) = \frac{\alpha}{(1-\alpha)} \, F_-(\la)\,.
\end{equation}
A priori, we might consider to use any of the four pairs 
$( F_-\,, F_+)$, $(\wt F_-\,,\wt F_+)$,   $(F_-\,,\wh F_+)$ and $(\wt F_-\,,\wh{\wt F}_+)$
for defining weights $f(x,y)$ on the edges in a way which remains invariant under
$\Aff(\T_q)$. But $F_-(\la)\,  \wh F_+(\la) = \wt F_-(\la) \wh{\wt F}_+(\la) = -1$,
and this is not compatible with \eqref{eq:ne1}. 

Thus, we have the natural choice $(F_-\,,F_+)$ and the ``twin'' $(\wt F_-\,,\wt F_+)$.
The weights provided by $\bigl(F_-(\la),F_+(\la)\bigr)$ in the sense of Section \ref{sec:intrep}
are the Green weights, $f(x,y) = F_{\pm}(\la)$ for neighbours $x,y$ with 
$\hor(y) = \hor(x)\pm 1$. An easy consequence of \eqref{eq:ug} is
$$
G(x,x|\la) = G(\la) = \frac{2q/\la}{(q-1) + (q+1)\sqrt{1 - 4\alpha(1-\alpha)/\la^2}}\,.
$$
We remark that from this one can deduce by classical spectral methods that
$\spec(P) = [-\rho\,,\,\rho]$, where $\rho = 2\sqrt{\alpha(1-\alpha)}$. Namely, $G(x,x|\la)$ is
the \emph{Stieltjes transform} of the \emph{Plancherel} or \emph{spectral measure}, 
also called  \emph{KNS-measure} by {\sc Grigorchuk and \.Zuk}~\cite{GZ}. 
That measure is the diagonal element
of the resolution of the identity of the operator $P$; in the context of infinite graphs, see
e.g. {\sc Mohar and Woess}~\cite{MoWo}. Some more details will be considered
in \S \ref{sec:invariance}. The measure, and in this case,
its density with respect to Lebesgue measure, can be computed via the
inversion formula of Stieltjes--Perron; see {\sc Wall}~\cite{Wall}. The spectrum is the
support of that measure. 

We also observe that our random walk is $\rho$-transient precisely when $q \ge 2$.
We see that the Green weights
fulfill the requirements \eqref{eq:ne1} -- \eqref{eq:la} for any 
$\la \in \C \setminus [-\rho\,,\,\rho]$, as well as for $\la=\pm\rho$ when $q \ge 2$.

On the other hand, the only value of $\la$ for which \eqref{eq:ne1} does not hold, i.e. $\wt F_-(\la) \wt F_+(\la) = 1$, is
\begin{equation}\label{eq:la0}
\la_0 = \frac{q+1}{2\sqrt{q}}\,\rho = \frac{\rho}{\rho(\text{SRW})}\,,
\end{equation}
where $\rho(\text{SRW})$ is the spectral radius of the simple random walk on $\T_q\,$,
that is the random walk that arises for $\alpha  = q/(q+1)$. It is also easy to check that
$$
\wt U(\la) = \wt U(x,x|\la) := \sum_{y \sim x} p(x,y) \wt F(y,x|\la) = 
\frac{q+1}{2q}\la \Bigl( 1 + \sqrt{1 - 4\alpha(1-\alpha)/\la^2}\Bigr)
$$
satisfies $\wt U(\la) = \la$ precisely when $\la = \la_0\,$.
Thus, using $\bigl(\wt F_-(\la),\wt F_+(\la)\bigr)$, the weights 
$\tilde f(x,y) = \wt F_{\pm}(\la)$ for $x \sim y$ with 
$\hor(y) = \hor(x)\pm 1$ 
fulfill the requirements \eqref{eq:ne1} -- \eqref{eq:la} for any 
$\la \in \C \setminus (-\rho\,,\,\rho)$, with the exception of $\la_0\,$. 

According to  \eqref{eq:F}, resp~ \eqref{eq:f}, for $\la \in \C \setminus [-\rho\,,\,\rho]$
and arbitrary $x,y \in T$ we have the extensions
$$
F(x,y|\la) = F_-(\la)^{d(x,v)}F_+(\la)^{d(v, y)}
\AND
\wt F(x,y|\la) = \wt F_-(\la)^{d(x,v)}\wt F_+(\la)^{d(v, y)}\,,
$$
where $v$ is the unique point in $\pi(x,y)$ where $\hor(v)$ attains its minimum along that geodesic.
The associated Poisson kernels are
\begin{equation}\label{eq:KtildeK}
\begin{aligned}
K(x,\xi|\la) &=  F_-(\la)^{\hor(x,\xi)}\biggl(\frac{(1-\alpha)q}{\alpha}\biggr)^{\ell(x,\xi)} 
\;\quad\text{and}\\
\wt K(x,\xi|\la) &=  
\wt F_-(\la)^{\hor(x,\xi)}\biggl(\frac{(1-\alpha)q}{\alpha}\biggr)^{\ell(x,\xi)},\quad \text{where}\\
\ell(x,\xi) &= d\bigl(x \wedge \xi, \pi(o,\varpi)\bigr).
\end{aligned}
\end{equation}
This formula arises as follows: first, $\ell(x,\varpi)=0$ so that
$K(x,\varpi) = F_-(\la)^{\hor(x)}$. If $\xi \ne \varpi$ then let $v = \varpi \wedge \xi$
and $c = x \wedge \xi \in \pi(v,\varpi) \cup \pi(v,\xi)$; see Figure~1. 
If $c \in \pi(v,\varpi)$, then $F(x,c|\la) = F_-(\la)^{d(x,c)}$,
$F(o,c|\la) = F_-(\la)^{d(o,c)}$, and  $\ell(x,\xi) = 0$.
On the other hand, if $c \in \pi(v,\xi)$ then still $F(x,c|\la) = F_-(\la)^{d(x,c)}$,
but $F(o,c|\la) = F_-(\la)^{d(o,v)}F_+(\la)^{d(v,c)}$, and  $d(v,c)=\ell(x,\xi)$. Now the first identity in \eqref{eq:KtildeK} follows from
\eqref{eq:sol}.
The same arguments apply to $\wt K$ and $\wt F$. 
Note that, for $\la = \pm \rho$, we have $\wt K = K$.

\begin{rmk}\label{obs:Z} Consider the case when $q=1$ and the random walk is on $\T_2 \equiv \Z$. 
Its  non-zero transition probabilities are 
$$
p(x,x+1) = \alpha \AND p(x,x-1) = 1-\alpha\,, \quad x \in \Z\,.
$$
Then it is natural to write $\partial \T_2 = \partial\Z = \{\pm \infty\}$, with $\varpi = -\infty$.
Note that $\la_0=\rho$. When $\la \in \C \setminus [-\rho\,,\,\rho]$, we have 
$$
\wt K(x,+\infty|\la) = K(x,-\infty|\la) \AND \wt K(x,-\infty) = K(x,+\infty|\la)\,,
$$
the kernels at $+\infty$ and $-\infty$ are distinct, and every $\la$-eigenfunction
arises as a unique linear combination of those two kernels.

When $\la = \rho$, the function $K(x, +\infty|\rho)= K(x, -\infty|\rho)$ is the unique
positive $\rho$-harmonic function with value $1$ at the origin. 
\hfill $\square$
\end{rmk}

This settles the special case $q=1$. We are more interested in $q \ge 2$, where we get
the following.

\begin{cor}\label{cor:rep} For $q \ge 2$,
let $\la \in \C \setminus (-\rho\,,\,\rho)$, and let $h$ be a $\la$-harmonic function.
Then there is a unique strong distribution $\nu^h$ on $\partial T$
such that
$$
h(x) = \int_{\partial T} K(x,\xi|\la)\, d\nu^h(\xi)\,.
$$
If in addition $\la  \ne \la_0$ then there also is a unique strong distribution $\tilde \nu^h$
on $\partial T$ such that
$$
h(x) = \int_{\partial T} \wt K(x,\xi|\la)\, d\tilde\nu^h(\xi)\,.
$$
\end{cor}

Of course, when $\la = \pm \rho$, we have $\wt K = K$ and $\tilde \nu^h = \nu^h$,
but otherwise we shall see that the kernels and the representing distributions are distinct.

To our knowledge, this twin representation of $\la$-harmonic functions was first observed
and used for the simple random walk in the context of the representation theory of 
free groups by {\sc Mantero and Zappa}~\cite{MaZa}. 

If $\la \ge \rho$, then it is well-known that the functions $x \mapsto K(x,\xi|\la)$,
$\xi \in \partial T$, are the \emph{minimal} $\la$-harmonic functions, that is, the extremal 
elements of the convex set 
\begin{equation}\label{eq:base}
\mathcal{H}_o(\la) = \{ h: T \to (0\,,\,\infty) \mid h(o)=1\,,\; Ph = \la\cdot h \}\,.
\end{equation}
(When $T$ is locally finite, this set is compact in the topology of pointwise convergence.)
The index $o$ stands for normalization at the reference point $o$.

\begin{thm}\label{thm:twin} Assume that $q \ge 2$.
For $\la \in \C \setminus [-\rho\,,\,\rho]$,  $\la \ne \la_0\,$,
and for $\xi \in \partial T$, let $\nu^{\xi}$ and $\tilde \nu^{\xi}$
be the strong distributions on $\partial T$ in the sense  of Corollary \ref{cor:rep} such that
$$
\wt K(x,\xi|\la) = \int_{\partial T} K(x,\cdot|\la) \,d\nu^{\xi}
\AND K(x,\xi|\la) = \int_{\partial T} \wt K(x,\cdot|\la) \,d\tilde\nu^{\xi}.
$$ 
Then $\nu^{\xi}$ extends to a complex ($\sigma$-additive) Borel measure on $\partial T$,
while this does not hold for $\tilde \nu^{\xi}\,$.

If, in particular, $\la > \rho$ is real, then the Borel 
probability measure $\nu^{\xi}$ is supported 
by all of $\partial T$, so that $\wt K(\cdot,\xi|\la)$ is not minimal in  $\mathcal{H}_o(\la)$.
\end{thm}

\begin{proof} We start with an inequality that will be needed below:
\begin{equation}\label{eq:ineq}
\biggl|\frac{F_-(\la)}{\wt F_-(\la)}\biggr| 
= \biggl|\frac{1- \sqrt{1-\rho^2/\la^2}}{\rho/\la}\biggr|^2 < 1
\quad \text{for all }\; \la \in \C \setminus [-\rho\,,\,\rho]\,. 
\end{equation}
Recalling Convention \ref{con:vention}, we obtain \eqref{eq:ineq} by a few elementary computations.

\smallskip

Now let $x \in T \setminus \{o\}$. Noting that $\hor(x^-) = \hor(x)\pm 1$, 
we can use the first ones of the respective
identities \eqref{eq:sol} and \eqref{eq:soltilde} plus \eqref{eq:eq1} to compute
\begin{equation}\label{eq:1overq}
\begin{aligned}
F(x,x^-|\la)\, &\wt F(x^-,x|\la) = F(x^-,x|\la)\, \wt F(x,x^-|\la) \\
&= F_-(\la) \wt F_+(\la) = F_+(\la) \wt F_-(\la) 
= F_-(\la)\,\wt F_-(\la) \,\frac{\alpha}{(1-\alpha)q} = \frac{1}{q}\,,
\end{aligned}
\end{equation}
because either $F(x,x^-|\la) = F_-(\la)$ and $\wt F(x^-,x|\la) = \wt F_+(\la)$,
or $F(x,x^-|\la) = F_+(\la)$ and $\wt F(x^-,x|\la) = \wt F_-(\la)$. In particular,
\begin{equation}\label{eq:1q}
 |F_+(\la)F_-(\la)| =  \left| \frac{F_-(\la)}{q \,\wt F_-(\la)}\right| < \frac{1}{q}\,.
\end{equation}

By Theorem \ref{thm:general}, 
$$
\nu^{\xi}(\partial T_x) = F(o,x|\la)\,
\frac{\wt K(x,\xi|\la) -F(x,x^-|\la)\wt K(x^-,\xi|\la)}{1-F_+(\la)F_-(\la)}
$$

\noindent\emph{\underline{Case 1:}} $x \in \pi(o,\xi)$. 

\smallskip

Then $\wt K(x,\xi|\la) = 1/\wt F(o,x|\la)$ and  
$\wt K(x^-,\xi|\la) = \wt F(x^-,x|\la)/\wt F(o,x|\la)$, and
\eqref{eq:1overq} yields
\begin{equation}\label{eq:case1}
\nu^{\xi}(\partial T_x) = \dfrac{1-1/q}{1-F_+(\la)F_-(\la)} \, 
\dfrac{F(o,x|\la)}{\wt F(o,x|\la)}\, 
 = \dfrac{1-1/q}{1-F_+(\la)F_-(\la)}\,
\biggl(\dfrac{F_-(\la)}{\wt F_-(\la)}\biggr)^{d(o,x)}\,.
\end{equation}
We note immediately that this is strictly positive when $\la > \rho$, because
in view of \eqref{eq:FU}, combined with \eqref{eq:FtildeF} and \eqref{eq:sol}, 
we then have $F_+(\la)F_-(\la) \le F_+(\rho)F_-(\rho) 
= \bigl(\frac{\rho}{2\al}\bigr)^2 \frac{\al}{(1-\al)q} = \frac{1}{q} < 1$.
\\[6pt]
\emph{\underline{Case 2:}} $x \notin \pi(o,\xi)$. 

\smallskip

Let $v = x \wedge \xi = x^- \wedge \xi$.  Then 
$\wt K(x,\xi|\la) = \wt F(x,v|\la)/\wt F(o,v|\la)$ and  
$\wt K(x^-,\xi|\la) = \wt K(x,\xi|\la)/\wt F(x,x^-|\la)$.
Now, \eqref{eq:1overq} yields that $F(v,x|\la) \wt F(x,v|\la) = q^{-d(x,v)}$,
because it is the product of $d(x,v)$ terms of the form 
$F(x_i^-\,,x_i|\la) \wt F(x_i\,,x_i^-|\la)$. Therefore 
\begin{equation}\label{eq:case2}
\begin{aligned}
\nu^{\xi}(\partial T_x) &= 
\overbrace{\frac{1-F_-(\la)/\wt F_-(\la)}{1-F_+(\la)F_-(\la)}}^{\displaystyle =:C(\la)}\,
F(o,x|\la) \,\frac{\wt F(x,v|\la)}{\wt F(o,v|\la)}\\
&= C(\la)\,\frac{F(o,v|\la)}{\wt F(o,v|\la)}\,F(v,x|\la) \wt F(x,v|\la)
= C(\la)\, \biggl(\frac{F_-(\la)}{\wt F_-(\la)}\biggr)^{d(o,v)}\biggl(\frac{1}{q}\biggr)^{d(x,v)}\,.
\end{aligned}
\end{equation}
Again, this is strictly positive when $\la > \rho$, and we obtain that
in this case the Borel probability measure $\nu^{\xi}$ is supported by all of
$\partial T$.

\smallskip

We now prove that for any $\la \in \C \setminus (-\rho\,,\rho)$, the distribution 
$\nu^{\xi}$ extends to a $\sigma$-additive Borel measure on $\partial T$.
Let $(x_n)_{n \ge 0}$ be a sequence of vertices such that the arcs $\partial T_{x_n}$ are
pairwise disjoint. 

Write $\pi(o,\xi)=[o=v_0\,,v_1\,,\dots]$. There can be at most
one $x_n$ on that geodesic ray. 
In that case, suppose it is $x_0\,$, that is, $x_0 = v_k$ for some $k \ge 0$.
By  \eqref{eq:ineq}, \eqref{eq:1q} and \eqref{eq:case1},
$$
|\nu^{\xi}(\partial T_{x_0})| <  \biggl|\dfrac{1-1/q}{1-F_+(\la)F_-(\la)}\biggr|< 1\,.
$$
Next, let $A_k = \{ n \ge 1: x_n \wedge \xi = v_k \}$. We claim that, using 
\eqref{eq:case2}, one has
$$
\sum_{n\,:\, x_n \in A_k} \bigl|\nu^{\xi}(\partial T_{x_n})\bigr| \le  
|C(\la)|\cdot \bigl|F_-(\la)/\wt F_-(\la)\bigr|^k\,.
$$
Indeed, consider the equidistribution $\ol\nu$ on $\partial T$, that is, 
$\ol\nu(\partial T_x) = 1\big/\bigl((q+1)q^{d(o,x)-1}\bigr)$ for $x \ne 0$. 
It extends to a Borel probability measure on $\partial T$, and for $k \ge 1$, 
$$
\sum_{n\,:\, x_n \in A_k} \!\!q^{-d(x_n,v_k)} 
= (q+1)q^{k-1} \!\!\sum_{n\,:\, x_n \in A_k}\!\! \ol\nu(\partial T_{x_n}) 
\le (q+1)q^{k-1} \,\,\ol\nu(\partial T_{v_k} \setminus \partial T_{v_{k+1}}) = \frac{q-1}{q}\,.
$$
For $k = 0$, the analogous computation yields the upper bound $1$.
By \eqref{eq:ineq}, 
$$
\sum_{n=0}^{\infty}\bigl|\nu^{\xi}(\partial T_{x_n})\bigr| \le 
\biggl|\dfrac{1-1/q}{1-F_+(\la)F_-(\la)}\biggr| \,+\,
 \sum_{k=0}^{\infty} \sum_{n\,:\, x_n \in A_k}\bigl|\nu^{\xi}(\partial T_{x_n})\bigr|
\le 1 + \frac{|C(\la)|}{1-|F_-(\la)/\wt F_-(\la)|}\,.
$$
So condition \eqref{eq:M} is satisfied, and $\nu^{\xi}$ has a $\sigma$-additive extension,
as stated.

\smallskip

To obtain the analogous formulas to \eqref{eq:case1} and \eqref{eq:case2} for
$\tilde\nu^{\xi}$, we just have to exchange $F$ and $\wt F$ in each occurence.
We write $\wt C(\la)$ for the resulting constant in the analogue of \eqref{eq:case2}.
In this case, let the sequence $(x_n)$ consist of all the neighbours of the $v_k\,$, $k \ge 1$ 
which do not lie on $\pi(o,v)$. Thus, the set $A_k$ defined above consists of the neighbours of
$v_k$, and by the same computation  we obtain
$$
\sum_{n\,:\, x_n \in A_k} \bigl|\tilde\nu^{\xi}(\partial T_{x_n})\bigr| =  
|\wt C(\la)|\cdot \bigl|\wt F_-(\la)/ F_-(\la)\bigr|^k\,.
$$
The sum over all $k$ diverges by \eqref{eq:ineq}, so that $\wt \nu^{\xi}$ does not 
satisfy the bounded variation condition \eqref{eq:M}.
\end{proof}

\section{General transitive group actions}\label{sec:invariance}

After the detailed study of multiple integral representations in
\S \ref{sec:twin}, we now turn to general transitive group actions in the
place of $\Aff(\T_q)$. Once more, we take up material from our ``companion'' paper
\cite[\S 4]{PiWo}: we assume that the transition probabilities are invariant
under a general group $\Gamma$ of automorphisms of the tree which acts transitively 
on the vertex set. That is, 
$$
p(\gamma x,\gamma y) = p(x,y) \quad\text{for all }\; x,y \in T \; \text{ and }\;\gamma \in \Gamma\,.
$$
Let $I = \Gamma \backslash E(T)$ be the set of orbits of $\Gamma$ on the set of oriented
edges of $T$. If $j \in I$ is the orbit \emph{(type)} of $(x,y) \in E(T)$ then we write 
$p_j=p(x,y)$ and $-j$ for the orbit of $(y,x)$. Then $-j$ is independent of the representative 
$(x,y)$, and $-(-j) = j$. In particular, $-j = j$ if and only if there is 
$\gamma \in \Gamma$ for which $\gamma x =y$ and $\gamma y = x$. 
For each $j \in I$ and fixed $x \in T$, we set $d_j = |\{ y \sim x : (x,y) \in I \}|$.
This is finite because $d_j \le 1/p_j\,$, and independent of $x$ by transitivity of $\Gamma$. 
For example, when $\Gamma = \Aut(\T_q)$ then $I = \{1\}$ with $d_1 = q+1$, while when 
$\Gamma = \Aff(\T_q)$ then $I = \{ \pm 1 \}$ with $d_{-1} = 1$ and $d_{1} = q$.
Thus, $\sum_{j \in I} d_jp_j =1$, and  $\dg(x) = \sum_{j \in I} d_j\,$.

As clarified in \cite[Remark 4.4, second half]{PiWo}, one can start with a finite or countable set $I$ 
with an involution  $j \mapsto -j$ and a collection $(d_j)_{j \in I}$ of natural numbers.
Then for  the regular tree $T$ with degree $\sum_j d_j \le \infty\,$, there is    
a group $\Gamma \le \Aut(T)$ which acts transitively and such that $I$ is is in one to one correspondence with its  set
of orbits and the associated cardinalities are $d_j$.

For example, when $d_j=1$ 
for all $j$, then we can choose $\Gamma$ as the discrete group
\begin{equation}\label{eq:discgp}
\Gamma = \langle a_j, j \in I \mid 
a_j^{-1} = a_{-j} \;\text{for all}\; j \in I \rangle\,.
\end{equation}
Then, when $j \ne -j$, we can choose just one out of $a_j$ and
$a_{-j}$ as a free generator. Instead, when $j = -j$, then $a_j$
is a generator whose square is the group identity. In this example, $\Gamma$ acts transitiviely
with trivial stabilizers, and the fact that this provides all possible groups which act 
in this way on a countable tree is a well-known basic part of Bass--Serre theory (see {\sc Serre}~\cite{Se}). In all other cases, $\Gamma$ will have non-discrete closure in
$\Aut(T)$.

In the general situation of a transitive group action which leaves the transition probabilities
invariant, it is shown in \cite[Thm.\;4.2]{PiWo} that $\res(P) \setminus \res^*(P) \subset \{0\}$.
That is, $G(\la) \ne 0$ for all $\la \in \res(P) \setminus \{ 0\}$, where $G(\la) = G(x,x|\la)$,
which is independent of $x$ by transitivity. We remark that it may happen that $0$ 
is part of the resolvent set of $P$ \cite{FiSt}.

Here, we shall always assume that the vertex degree is $\ge 3$, so that our random walk
has to be $\rho$-transient by a result of {\sc Guivarc'h}~\cite{Gu}. When $I$ is infinite
we make the additional assumption that 
\begin{equation}\label{eq:ass}
\sum_{j \in I} d_j\, p_{-j} < \infty\,.
\end{equation}
Note that this is the sum over all neighbours of any vertex $x$ of the incoming probabilities 
$p(y,x)$. The assumption is satisfied, for instance, if the quotients $p_{-j}/p_j$ are bounded. 

If $(x,y)$ is an edge of type $j$, then $g(x,y) = G(x,y|\la) = G_j(\la)$ 
depends only on $j$. By reversibility, we have
$$
p_j\, G_{-j}(\la) = p_{-j}\, G_j(\la)\,,
$$
and the second identity of Lemma \ref{lem:recover}  
becomes 
\begin{equation}\label{eq:GG}
p_{-j}\, G_j(\la)^2 + G_j(\la) - p_j\, G(\la)^2 = 0\,.  
\end{equation}
When $\la > \rho$ is real, among the two solutions of this equation the meaningful one is
\begin{equation}\label{eq:Gj}
G_j(\la) = \frac{1}{2p_{-j}}\Bigl(\sqrt{1 + 4p_jp_{-j} \,G(\la)^2} - 1\Bigr)\,,
\end{equation} 
because the functions $G(\la)$ and $G_j(\la)$ are decreasing in this range of $\la$.
In other regions of the plane, there may be a minus sign in front of the root.

\begin{pro}\label{pro:eqext} Let $\kappa = \max \{ 2\sqrt{p_j p_{-j}} : j \in I \}$.
Then the identity \eqref{eq:Gj} holds for all $\la$ 
in the set
$$
\mathcal{U} = \{ \la \in \C : |\la| > \rho \} \setminus \{ \pm i\,t : \rho < t \le \kappa \}\,.
$$
(When $\kappa \le \rho$ the last part is empty.)  
\end{pro}
 
\begin{proof} Each of the functions 
\begin{equation}\label{eq:Phij}
\Phi_j(t) = \frac{1}{2}\Bigl(\sqrt{1 + 4p_jp_{-j} \,t^2} - 1\Bigr)
\end{equation}
is analytic in the slit plane 
\begin{equation}\label{eq:slit}
\mathcal{W} = \C \setminus \{ \pm i\, t : t \ge 1/\kappa\}\,.
\end{equation}
We shall show that the function $G(\la)$ maps $\mathcal{U}$ into $\mathcal{W}$.
This implies that the functions appearing in \eqref{eq:Gj}
are all analytic, so that the identity must hold on all of $\mathcal{U}$ by analytic
continuation.

We use some well-known spectral theory. Let $\mu$ be the Plancherel measure
of our random walk, introduced  in \S \ref{sec:twin}. Recall that
$\mu$ is a probability measure concentrated on $\spec(P)$, and
is the diagonal matrix element at $(x,x)$ (independent of $x \in T$ by group invariance)
of the spectral resolution of the self-adjoint operator $P$ on $\ell^2(T,\mm)$.
In more classical terms, it is the measure on $[-\rho\,,\,\rho]$ whose 
$n^{\textrm{th}}$ moments are the return probabilities $p^{(n)}(x,x)$ for $n \ge 0$. 
Since in the present case, these probabilities are $0$ when $n$ is odd, $\mu$ 
is symmetric (invariant under the reflection $t \mapsto -t$). Thus 
$$
G(\la) = \int_{[-\rho\,,\,\rho]} \frac{1}{\la - t}\; d\mu(t)\,,\quad 
\la \in \C\setminus \spec(P)\,.
$$
Now let $|\la| > \rho$ be such that $\Re(\la) \ne 0$, and write $\bar \la$ for 
its complex conjugate. Then
$$
\begin{aligned}
\overline{G(\la)} &= \int_{[-\rho\,,\,\rho]} \frac{1}{\bar\la - t}\; d\mu(t)
= \int_{[-\rho\,,\,\rho]} \frac{1}{\bar\la + t}\; d\mu(t)\,,\quad \text{whence}\\[4pt]
\Re \bigl(G(\la)\bigr) &= \frac12 \int_{[-\rho\,,\,\rho]}\biggl(\frac{1}{\la - t} + 
\frac{1}{\bar\la + t}\biggr) \;d\mu(t)
= \Re(\la)  \int_{[-\rho\,,\,\rho]} \frac{|\la|^2 - t^2}{(|\la|^2 - t^2)^2 + 4t^2\Im(\la)^2}\;
\,d\mu(t).
\end{aligned}
$$
The last integral is $> 0$, so that also $\Re \bigl(G(\la)\bigr) \ne 0$. 
Therefore $G(\la) \in \mathcal{W}\,.$

Next, let $\la = i\,\beta$, where $\beta \in \R$ and $|\beta| > \max \{\rho, \kappa\}$. 
Then, using again that $\mu$ is symmetric (so that odd functions integrate to $0$), 
$$
G(i\,\beta) = \int_{[-\rho\,,\,\rho]} \frac{-i\, \beta - t}{\beta^2 + t^2}\,d\mu(t) 
= -\frac{i}{\beta}\int_{[-\rho\,,\,\rho]} \frac{1}{1 + (t/\beta)^2}\,d\mu(t)\,. 
$$
Therefore $|G(i\,\beta)| \le 1/|\beta| < 1/\kappa$, and also $G(i\,\beta) \in \mathcal{W}\,$.
\end{proof}

We now obtain the following.

\begin{thm}\label{thm:phi} For $\la \in \mathcal{U}$,
$$
\la \,G(\la) = \Phi\bigl( G(\la) \bigr)\,,\quad \text{where}\quad
\Phi(t) = 1 + \sum_{j \in I}  \frac{d_j}{2}\Bigl(\sqrt{1 + 4p_jp_{-j} \,t^2} - 1\Bigr)\,.
$$
The function $\Phi(t)$ is analytic in the domain $\mathcal{W}$ of \eqref{eq:slit}. Furthermore,
$$\rho = \min \{ \Phi(t)/t : t > 0 \} = \Phi(\theta)/\theta\,,
$$
where $\theta$
is the unique positive real solution of the equation $\Phi'(t) = \Phi(t)/t$.
\end{thm}

\begin{proof}
First of all, observe that for $t \in \C \setminus \{ i\,s : s \in \R\,,\; |s| \ge 1\}$,
$$
\big| \sqrt{1 + t^2} - 1 \big| < |t|\,.
$$
Therefore, summing over all $j \in I$,
$$
\sum \frac{d_j}{2}\Bigl|\sqrt{1 + 4p_jp_{-j} \,t^2} - 1\Bigr|
< |t| \sum d_j \sqrt{p_jp_{-j}} \le |t| \sqrt{\sum d_j\,p_{-j}}\,, 
$$
which is finite by assumption \eqref{eq:ass}. Consequently, even when
$I$ is infinite, the defining series of $\Phi(t)$ converges absolutely and
locally uniformly on $\mathcal{W}$, so that $\Phi(t)$ is indeed analytic
on that set. Now we can use \eqref{eq:Gj} and Proposition \ref{pro:eqext}:
for $\la \in \mathcal{U}$,
$$
\la\, G(\la) - 1 = \sum_y p(x,y)G(y,x|\la) = \sum_{j \in I} d_j\,p_j\,G_{-j}(\la) 
= \Phi\bigl(G(\la)\bigr) - 1\,.
$$
The remaining statements of the theorem follow well-known lines, compare e.g.
with \cite[Ex. 9.46]{W-Markov}, where the variable $z=1/\la$ is used instead of $\la$, 
and see also below.  
\end{proof}

\begin{rmks}\label{rmks:old}
For the free group with (finitely or) infinitely many generators, the equation
for $G(\la)$ of Theorem \ref{thm:phi} was first deduced and used for finding
the asymptotics of $p^{(n)}(x,x)$ by {\sc Woess}~\cite{Wfree}. Its validity
was restricted to a complex neighbourhood of the real half-line $[\rho\,,+\infty)$  
There, computations are performed in the variable $z=1/\la$. 
A previous variant (for $z$, resp. $\la$
positive real) is inherent in work of {\sc Levit and Molchanov}~\cite{LeMo}. 
Later on, {\sc Aomoto}~\cite{Ao} considered equations of the same nature
as \eqref{eq:GG} for the case of finitely generated free groups plus reasonings of
algebraic geometry to study the nature of the involved functions and the
spectrum of $P$. Similarly, {\sc Fig\`a-Talamanca and Steger}~\cite{FiSt} considered the case 
when the group is discrete as in \eqref{eq:discgp}, $I$ is finite, and $j=-j$ for 
all $j$. This served for an in-depth study of the associated harmonic analysis.

What is new here is 
\begin{itemize}
\item the extension to the general group-invariant case,
with $I$ finite or infinite, 
\item
 the validity of the equation for $G(\la)$
in the large domain $\mathcal{U}$. 
\end{itemize}
This domain can be further extended a bit by
additional estimates, but for complex $\la$ close to $\spec(P)$, the situation is
more complicated. Indeed, in such regions, the correct 
solution of \eqref{eq:GG} may be the one where one has to use the negative branch 
of the square root in \eqref{eq:Gj}.
The general formula instead of the one of Theorem \ref{thm:phi} is then
$$
\la\, G(\la) = 1 +  \sum_{j \in I}  \frac{d_j}{2}\Bigl(\pm\sqrt{1 + 4p_jp_{-j} \,G(\la)^2} - 1\Bigr),
$$
where the signs may vary according to the region to which $\la$ belongs.
This requires
some subtle algebraic geometry beyond the focus of the present paper \cite{Ao}, \cite{FiSt}.
\end{rmks}

In the general group-invariant set-up, and even for non-locally finite $T$,
we obtain the integral representation 
of Theorem  \ref{thm:general} with respect to the Martin kernel 
$k(x,\xi) = K(\cdot, \cdot|\la)$ for any 
$\la$-harmonic function, whenever $0 \ne \la \in \C \setminus \spec(P)$, for $\la = \pm \rho$, 
and possibly also for $\la =0$. 

The study of twin kernels and the resulting integral
representation of $\la$-harmonic functions becomes more delicate in view of the
fact that $G(\la)$, and thus also the functions $F_j(\la) = G_j(\la)/G(\la)$,
are only given via the implicit equation for $G(\la)$ of Theorem \ref{thm:phi}. Therefore we limit attention to the case when $\la \in (\rho\,,\,+\infty)$ is real. For real $t$,
each function $\Phi_j$ of \eqref{eq:Phij} describes the upper branch of a hyperbola.
Thus, the function $\Phi$ has the following properties: it is strictly increasing and 
strictly convex, 
$$
\Phi(0)=1\,,\quad \Phi'(0)=0\,,\AND \lim_{t \to \infty} \Phi'(t) = \la_0\,,
\; \text{ where } \la_0 = \sum_{j \in I} d_j \sqrt{p_jp_{-j}}\,.
$$
We have $\la_0 < \infty$ by assumption \eqref{eq:ass}. Note that in the case of the 
affine random walks of \S \ref{sec:twin}, this is the same $\la_0$ as in \eqref{eq:la0}.
\begin{figure}[ht]
$$
\beginpicture  
\setcoordinatesystem units <.8mm,.8mm> 

\setplotarea x from -13 to 66, y from -12 to 61

\arrow <1.5mm> [.2,.6] from -2 0 to 60 0
\put {$t$} [l] <1mm, 0mm> at 60 0

\arrow <1.5mm> [.2,.6] from 0 -7 to 0 58
\put {$y$} [b] <0mm, 1mm> at 0 58

\setlinear \plot -2 -7  60 55 /
\put {$y = \la_0\,t - \frac{q-1}{2}$} [l] <1mm, -.4mm> at 60 55

\setlinear \plot -2 -4  30 60 /
\put {$y =\la\,t$} [b] <0mm, .5mm> at 30 60

\setquadratic \plot -2 10.13  0 10  2 10.13  10 13.03  20 20  
                    30 28.54  40 37.72  50 47.01  60 56.85 /
\put {$y = \Phi(t)$} [br] <2mm, .7mm> at 60 56.9

\put {\rm Figure~2} at 31 -10

\endpicture
$$
\end{figure}
For $\la \ge \la_0\,$, the equation
$\la\, t = \Phi(t)$ has a unique positive solution. This is $t = G(\la)$. See Figure~2,
where we assume that $\dg(x) = q+1$ is finite. 
With $\theta$ and $\rho$ as in Theorem \ref{thm:phi}, it is clear from the shape
of $\Phi$ that $\la_0 > \rho$, and for $\la_0 > \la > \rho$, there are precisely
two solutions of the equation $\la\, t = \Phi(t)$. One is smaller than $\theta$ and the other is
larger than $\theta$. By continuity of $G(\cdot)$, the correct solution for $G(\la)$ is the one
for which $G(\la) < \theta$: this is the solution that leads to the ordinary $\la$-Martin kernel
$K(\cdot,\cdot|\la)$ and the resulting integral representation of any $\la$-harmonic
function over $\partial T$.   
But we also have the second solution $\wt G(\la) > \theta$. Working with this one,
we also find that for all $\la \in (\rho\,,\,\la_0)$ one has
$$
\wt G(\la) \ne 0\AND  \wt G_j(\la) = \Phi_j\bigl(G(\la)\bigr)\big/p_{-j} \ne 0\,,
$$
whence $\wt F_j(\la) = \wt G_j(\la)/\wt G(\la) \ne 0$.
Also,
$$
\wt F_j(\la) \,\wt F_{-j}(\la) = \frac{\Phi_j\bigl(\wt G(\la)\bigr)^2}{p_jp_{-j}\, \wt G(\la)^2} < 1\,,
$$
since $\sqrt{1+t^2}-1 < t$ for $t > 0$. 
Thus, \eqref{eq:ne1} holds for the weigths $f(x,y) = \wt F_j(\la)$, when $(x,y)$
is an oriented edge of type $j$. Let us verify \eqref{eq:nela}:
$$
\sum_v p(x,v) f(v,x) = \sum_{j \in I} d_j\, p_j\, \frac{\wt G_{-j}(\la)}{\wt G(\la)} 
= \frac{\Phi\bigl(\wt G(\la)\bigr) - 1}{\wt G(\la)} = \frac{\la\,\wt G(\la)-1}{\wt G(\la)} < \la\,.
$$
Finally, \eqref{eq:la} reduces to equation \eqref{eq:GG}, which holds for $\wt G_j(\la)$
as well as for $G_j(\la)$.  We conclude that these edge weights lead to a second kernel
$$
k(x,\xi) = \wt K(x,\xi|\la)\,,\quad x \in T\,,\; \xi \in \partial T\,,\quad 
\la \in (\rho\,,\la_0)\,,
$$
so that $x \mapsto \wt K(x,\xi|\la)$ is positive $\la$-harmonic. 
Thus, every $\la$-harmonic function has a second integral representation as in
Theorem \ref{thm:general}, in addition to the one with respect to the ordinary
Martin kernel $K(\cdot,\cdot|\la)$. 

Again, for any $\xi \in \partial T$, there is a positive ($\sigma$-additive$\,$!)
Borel probability measure $\nu^{\xi}$ on $\partial T$ such that
$$
\wt K(x,\xi|\la) = \int_{\partial T} K(x,\cdot|\la)\, d\nu^{\xi}.
$$
We omit the computation which shows that $\nu^{\xi}$ is supported by all of
$\partial T$, which is a consequence of the fact that $T$ has degree $\ge 3$.
In particular, $\wt K(x,\xi|\la)$ cannot be a minimal $\la$-harmonic function, i.e.,
an extremal point of the set $\mathcal{H}_o$ of \eqref{eq:base}.
Therefore the converse representing distribution $\tilde \nu^{\xi}\,$, that
by Theorem \ref{thm:general} gives the integral representation
$$
K(x,\xi|\la) = \int_{\partial T} \wt K(x,\cdot|\la)\, d\tilde \nu^{\xi}\,,
$$
cannot have a $\sigma$-additive extension.

\smallskip

We may ask how to proceed for $\la > \la_0$, while we exclude the case $\la =  \la_0\,$,
since we have already seen in \S \ref{sec:twin} that for affine random walks there is
no natural choice for a second family of weights for $\la_0\,$.
We choose to proceed as follows, \emph{requiring here that $I$ be finite} 
and $\sum_j d_j = q+1$. 

The second solution of \eqref{eq:GG} is
$$
\wt G_j(\la) = \frac{1}{p_{-j}} \;\wt \Phi_j\bigl(\wt G(\la) \bigr)\,,\quad \text{where}
\quad \wt \Phi_j(t) = \frac{1}{2}\Bigl(-\sqrt{1 + 4p_jp_{-j} \,t^2} - 1\Bigr).
$$
Then we set 
$$
\wt \Phi(t) = \sum_{j \in I} d_j\, \wt \Phi_j(t)\,.
$$
(When $I$ is infinite, the series does not converge.)
While $\Phi(t)$ is a sum of upper  branches of hyperbolic functions, $\wt \Phi(t)$ it the sum
of the associated lower branches. The two asymptotes of $\Phi(t)$ and $\wt \Phi(t)$
are  $\;y = \pm \la_0 t -(q-1)/2\,$. Thus, any line $y = \la\, t$ has exactly two
intersection points with the ``twin curve'' $\bigl(\Phi(t), \wt \Phi(t)\bigr)$, except for
$\la = \pm \rho$, in which cases there is only one double solution, and $\la = \pm \la_0$, 
in which case there is only one simple solution. Thus, for $\la > \la_0$, we choose $\wt G(\la)$
as the unique solution of 
$$
\la \, \wt G(\la) = \wt\Phi\bigl(\wt G(\la)\bigr),
$$ 
which is negative. The associated solution for $\wt G_j(\la)$ is 
$$
\wt G_j(\la) = \frac{1}{p_{-j}} \; \wt \Phi_j\bigl(\wt G(\la)\bigr)\,,
$$
so that indeed 
$$
\sum_{j \in I} d_j\, p_j\, \wt G_{-j}(\la) = \wt\Phi\bigl(\wt G(\la)\bigr) - 1
=  \la \, \wt G(\la) - 1
$$
Note that also $\wt G_j(\la) < 0$, so that $\wt F_j(\la) = \wt G_j(\la)/\wt G(\la) > 0$.
The associated edge weights are again given by $f(x,y) = \wt F_j(\la)$, when $(x,y)$
is an oriented edge of type $j$.
It is straightforward to see that they also satisfy the requirements \eqref{eq:ne1} -- \eqref{eq:la},
so that we also obtain a positive kernel $k(x,\xi) = \wt K(x,\xi|\la)$ with the same
properties as above.

By symmetry, analogous properties hold for
negative $\la \in (-\infty\,,\,-\rho) \setminus\{-\la_0\}\,$.

\end{document}